\documentclass[10pt,a4paper]{article}
\usepackage{charter}
\usepackage{graphicx}
\usepackage{amsmath,amsfonts,amssymb}
\usepackage{amsthm}
\usepackage{mathtools}
\usepackage[linesnumbered,norelsize,vlined,ruled]{algorithm2e}
\usepackage[textwidth=2.5cm,textsize=small]{todonotes}
\usepackage[inline]{enumitem}
\usepackage{cases}
\usepackage[affil-sl]{authblk}

\usepackage{dsfont}
\usepackage{nicefrac}
\usepackage{booktabs}
\usepackage{rotating}
\usepackage{url}
\usepackage[binary-units=true]{siunitx}
\usepackage{afterpage}
\usepackage{pdflscape}

\usepackage[colorlinks=true,linkcolor=blue,citecolor=blue,unicode]{hyperref}

\DeclareMathOperator{\conv}{conv}

\DeclareMathOperator{\rc}{rc}

\DeclareMathOperator{\obs}{Obs}

\DeclareMathOperator{\aff}{aff}

\DeclareMathOperator{\xc}{xc}

\newcommand{\R}{\mathds{R}}
\newcommand{\Z}{\mathds{Z}}

\newcommand{\Q}{\mathds{Q}}
\newcommand{\cO}{\mathcal{O}}
\newcommand{\bin}[1]{\{0,1\}^{#1}}

\newcommand{\I}{\mathcal{I}}

\newcommand{\vcompact}{v^\star_{\mathrm{com}}}
\newcommand{\vcut}{v^\star_{\mathrm{cut}}}
\newcommand{\vCG}{v^\star_{\text{CG}}}
\newcommand{\basic}{\textsf{basic} }
\newcommand{\dclosed}{\textsf{downcld} }
\newcommand{\sboxes}{\textsf{sboxes} }

\DeclarePairedDelimiter{\card}{\lvert}{\rvert}

\newcommand{\st}{:}		
\newcommand{\T}{^\intercal}
\newcommand{\sprod}[2]{{#1}\T{#2}}
\newcommand{\define}{\coloneqq}

\theoremstyle{plain}
\newtheorem{theorem}{Theorem}[section]
\newtheorem{lemma}[theorem]{Lemma}
\newtheorem{proposition}[theorem]{Proposition}

\newtheorem*{claim*}{Claim}

\newtheorem{observation}[theorem]{Observation}
\theoremstyle{definition}
\newtheorem{remark}[theorem]{Remark}

\begin{document}

\title{Efficient MIP Techniques for Computing the Relaxation Complexity}

\author[1]{Gennadiy Averkov}
\author[2]{Christopher Hojny}
\author[1]{Matthias Schymura}
\affil[1]{%
  BTU Cottbus-Senftenberg\\
  Platz der Deutschen Einheit 1\\
  03046 Cottbus, Germany\\
  \emph{email} \{averkov, schymura\}@b-tu.de
}
\affil[2]{%
  Eindhoven University of Technology\\
  Combinatorial Optimization Group\\
  PO Box~513\\
  5600 MB Eindhoven, The Netherlands\\
  \emph{email} c.hojny@tue.nl
}

\date{}

\maketitle

\begin{abstract}

The relaxation complexity $\rc(X)$ of the set of integer points $X$ contained in a polyhedron is the minimal number of inequalities needed to formulate a linear optimization problem over $X$ without using auxiliary variables.
Besides its relevance in integer programming, this concept has interpretations in aspects of social choice, symmetric cryptanalysis, and machine learning.

We employ efficient mixed-integer programming techniques to compute a robust and numerically more practical variant of the relaxation complexity.
Our proposed models require row or column generation techniques and can be enhanced by symmetry handling and suitable propagation algorithms.
Theoretically, we compare the quality of our models in terms of their LP relaxation values.
The performance of those models is investigated on a broad test set and is underlined by their ability to solve challenging instances that could not be solved previously.

\textbf{Keywords:} mixed-integer programming models, relaxation complexity,
branch-and-cut, branch-and-price
\end{abstract}

\section{Introduction}

Let~$X \subseteq \Z^d$ be such that~$X = \conv(X) \cap \Z^d$ and let~$Y
\subseteq \Z^d$.
A fundamental problem in various fields is to find a polyhedron~$P$ with the
minimum number of facets such that~$X \subseteq P$ and~$(Y \setminus X)
\cap P = \emptyset$.
We call this quantity the \emph{relaxation complexity of~$X$ w.r.t.~$Y$},
in formulae, $\rc(X,Y)$, and any such polyhedron a \emph{relaxation}.
In case~$Y = \Z^d$, we write~$\rc(X)$ instead of~$\rc(X, \Z^d)$.
In the theory of social choice, $X \subseteq \bin{d}$ can be interpreted as
the winning strategies of a simple game, see~\cite[Chap.~8.3]{TaylorPacelli2008}.
One is then interested in computing~$\rc(X, \bin{d})$, i.e., the smallest
number of inequalities needed to distinguish winning and loosing strategies.
In symmetric cryptanalysis, a subfield of cryptography, $\rc(X, \bin{d})$
corresponds to the minimum number of substitutions in symmetric key
algorithms~\cite{10.1007/978-3-662-45611-8_9}.
In machine learning, relaxations~$P$ correspond to polyhedral classifiers
that distinguish two types of data points~\cite{AstorinoGaudioso2002}.
The relaxation complexity is then the minimum size of a polyhedral
classifier.
Finally, of course, $\rc(X)$ is the minimum number of inequalities needed to formulate a linear optimization problem over~$X \subseteq \Z^d$ without using auxiliary variables.

Depending on the application, different strategies have been
pursued for computing and bounding the relaxation complexity.
For example, Kaibel \& Weltge~\cite{kaibelweltge2015lowerbounds} introduced
the notion of hiding sets for deriving lower bounds on~$\rc(X)$.
Using this technique, they could show that several sets~$X$
arising from combinatorial optimization problems have superpolynomial
relaxation complexity.
Moreover, $\rc(X,Y)$ can be found by computing the chromatic number of a
suitably defined hypergraph; deriving lower bounds on the chromatic number
allowed Kurz \& Napel~\cite{KurzNapel2016} to find a lower bound on~$\rc(X,
\bin{d})$ in the context of social choice.
In machine learning, algorithms have been devised to construct polyhedral
classifiers and thus providing upper bounds on~$\rc(X,Y)$, see~\cite{AstorinoGaudioso2002,DundarEtAl2008,pmlr-v13-manwani10a,OrsenigoVercellis2007}.
To find the exact value of~$\rc(X, \bin{d})$ in the context of symmetric
cryptanalysis, mixed-integer programming models have been investigated.
For higher dimensions, however, many of these models cannot compute~$\rc(X,
\bin{d})$ efficiently in practice.

In this article, we follow the latter line of research.
Given the relevance of knowing the exact value of~$\rc(X, Y)$, our aim is
to develop efficient mixed-integer programming (MIP) techniques
for computing~$\rc(X, Y)$, if both~$X$ and~$Y$ are finite.
More precisely, we investigate methods to compute~$\rc_\varepsilon(X,Y)$, a
more robust variant of~$\rc(X,Y)$ that is numerically more practical as we
discuss below.
To this end, we propose in Section~\ref{sec:basicModels} three different MIP models that
allow to compute~$\rc_\varepsilon(X,Y)$: a compact model as well as two more
sophisticated models that require row or column generation techniques.
Section~\ref{sec:modelComparison} compares the quality of the three models in terms of their
LP relaxation value, and we discuss several enhancements of the basic models
in Section~\ref{sec:modelEnhancements}.
These enhancements include tailored symmetry handling and propagation
techniques as well as cutting planes.
Finally, we compare the performance of the three different models on a
broad test set comprised of instances with different geometric properties
and instances arising in symmetric cryptanalysis (Section~\ref{sec:experiments}).
Our novel methods allow to solve many challenging instances
efficiently, which was not possible using the basic models.

We remark that the basic versions of two models have already been used by
us in~\cite{AverkovEtAl2021} to find~$\rc_\varepsilon(X,Y)$ for~$X$ being the integer
points in low-dimensional cubes and crosspolytopes.
These experiments helped us to prove general formulae for~$\rc(X)$ in these cases.
For this reason, we believe that the more sophisticated algorithms
described in this article are not only of relevance for practical
applications, but also to develop hypotheses for theoretical results.
Our code is publicly available at github\footnote{\url{https://github.com/christopherhojny/relaxation_complexity}}.

\paragraph{Related Literature}
One of the earliest references on the relaxation complexity goes back to
Jeroslow~\cite{jeroslow1975ondefining} who showed the tight bound~$\rc(X,\bin{d}) \leq 2^{d-1}$, for any $X \subseteq \bin{d}$.
This result has been complemented by Weltge~\cite{weltge2015diss} who
showed that most~$X \subseteq \bin{d}$ have~$\rc(X, \bin{d}) \geq
\frac{2^d}{c \cdot d^3}$, for some absolute constant~$c>0$.
Moreover, hiding sets proposed by Kaibel \&
Weltge~\cite{kaibelweltge2015lowerbounds} provide a lower bound on~$\rc(X)$.
The bound given by hiding sets can be improved by computing the chromatic
number of a graph derived from hiding sets, see~\cite{AverkovEtAl2021}.
Regarding the computability of~$\rc(X)$, it has been shown
in~\cite{averkovschymura2020complexity} that there exists a proper subset~$\obs(X)$
of~$\Z^d \setminus X$ such that~$\rc(X) = \rc(X, \obs(X))$.
If~$\obs(X)$ is finite, they show that~$\rc(X, \obs(X))$, and
thus~$\rc(X)$, can be computed by solving a mixed-integer program.
They also provide sufficient conditions on~$X$ that guarantee~$\obs(X)$ to
be finite.
Moreover, they establish that~$\rc(X)$ is computable if~$d \leq 3$; for~$d =
2$, a polynomial time algorithm to compute~$\rc(X)$ is discussed
in~\cite{AverkovEtAl2021}.
In general, however, it is an open question whether~$\rc(X)$ is
computable.

One drawback of relaxations of~$X$ as defined above is that they might be sensitive to numerical errors.
If~$\sprod{a}{x} \leq \beta$ is a facet defining inequality of a relaxation of~$X$
that separates~$y \in \Z^d \setminus X$, then we only know~$\sprod{a}{y} >
\beta$.
Thus, slightly perturbing~$a$ might not separate~$y$ anymore.
To take care of this, we suggested in~\cite{AverkovEtAl2021} to add a
safety margin~$\varepsilon > 0$ to the separation condition.
That is, if~$\sprod{a}{x} \leq \beta$ is a facet defining inequality of a
relaxation of~$X$ with~$\|a\|_\infty = 1$ that separates~$y$, then we
require~$\sprod{a}{y} \geq \beta + \varepsilon$.
In this case, we say that~$y$ is~\emph{$\varepsilon$-separated from~$X$}.
Then, $\rc_\varepsilon(X)$ denotes the smallest number of facets of any
relaxation of~$X$ that satisfies the safety margin condition\footnote{Note
  that the definition in~\cite{AverkovEtAl2021} is different, but both
  concepts coincide if the value of~$\varepsilon$ is defined appropriately.
  We follow the definition provided here, because it simplifies the
  discussion in this article.
}.
We call such a relaxation an~$\varepsilon$-relaxation of~$X$.
Analogously to~$\rc(X,Y)$, we define~$\rc_\varepsilon(X,Y)$ to be the smallest number of inequalities needed to $\varepsilon$-separate~$X$ and~$Y \setminus X$.
As~$\varepsilon$-relaxations are more restrictive than relaxations,
$\rc_\varepsilon(X) \geq \rc(X)$ for each~$\varepsilon > 0$.
In contrast to~$\rc(X)$, however, we show in~\cite{AverkovEtAl2021} that for every finite and full-dimensional $X \subseteq \Z^d$
there is a finite set~$Y \subseteq \Z^d \setminus X$ such
that~$\rc_\varepsilon(X) = \rc_\varepsilon(X,Y)$.
Thus, $\rc_\varepsilon(X)$ is computable and the aim of this article is to
develop MIP techniques that allow to find~$\rc_\varepsilon(X,Y)$ efficiently.
In particular, if~$\varepsilon$ approaches~0, then~$\rc_\varepsilon(X)$
converges towards~$\rc_\Q(X)$, a variant of the relaxation complexity which requires the relaxations to be rational.
Further variations of~$\rc(X)$ in which the size of coefficients in facet defining inequalities are bounded are discussed in~\cite{Hojny2020,hojny2018strong}.

Besides finding relaxations of~$X$, another field of research aims to find
outer descriptions of~$P = \conv(X)$ to be able to use linear programming
techniques to solve optimization problems over~$X$.
Since~$P$ might have exponentially many facets, the concept of
extended formulations has been introduced.
Extended formulations are polyhedra~$Q \subseteq \R^{d + k}$ whose
projection onto~$\R^d$ yields~$P$.
The smallest number of facets of an extended formulation of~$P$ is its
extension complexity~$\xc(P)$.
We refer the reader to the surveys of Conforti et
al.~\cite{ConfortiCornuejolsZambelli2013} and
Kaibel~\cite{kaibel2011extended} as well as the references therein.
Extended formulations that allow to use integer variables have been
discussed, e.g., by Bader et al.~\cite{DBLP:journals/mp/BaderHWZ18},
Cevallos et al.~\cite{DBLP:conf/soda/CevallosWZ18}, and
Weltge~\cite[Chap.~7.1]{weltge2015diss}.
A combination of $\rc(X,\{0,1\}^d)$ and $\xc(\conv(X))$ has been studied by
Hrube\v{s} \& Talebanfard~\cite{hrubestalebanfard2021onthe}.

\paragraph{Basic Definitions and Notation}
Throughout this article, we assume that~$d$ is a positive integer.
The set~$\{1,\dots,d\}$ is denoted by~$[d]$, and we write~$e_1,\dots,e_d$
to denote the~$d$ canonical unit vectors in~$\R^d$.
Moreover, $\Delta_d = \{0, e_1, \dots, e_d\} \subseteq \R^d$ is
the vertex set of the standard simplex in~$\R^d$, and~$\lozenge_d = \{0,
\pm e_1, \dots, \pm e_d\} \subseteq \R^d$ denotes the integer points in
the~$d$-dimensional standard crosspolytope.
The affine hull of a set~$X \subseteq \R^d$ is denoted by~$\aff(X)$.

A set~$X \subseteq \Z^d$ is called \emph{lattice-convex} if~$X = \conv(X)
\cap \Z^d$.
For a lattice-convex set~$X \subseteq \Z^d$, we say that~$H \subseteq
(\aff(X) \cap \Z^d) \setminus X$ is a \emph{hiding set} if, for any distinct~$y_1,
y_2 \in H$, we have~$\conv(\{y_1, y_2\}) \cap \conv(X) \neq \emptyset$.
Kaibel \& Weltge~\cite{kaibelweltge2015lowerbounds} proved that the
cardinality of any hiding set is a lower bound on~$\rc(X)$.
The maximum size of a hiding set is denoted by~$H(X)$.
Moreover, if~$Y \subseteq \Z^d$, we say that~$H$ is a \emph{$Y$-hiding set} if~$H$ is a hiding set that is contained in~$Y$.
Analogously to~$H(X)$, $H(X,Y)$ denotes the maximum size of a~$Y$-hiding set.

\section{Mixed-Integer Programming Models to Compute $\rc_\varepsilon(X,Y)$}
\label{sec:basicModels}

In this section, we discuss three different mixed-integer programming
models to compute~$\rc_\varepsilon(X,Y)$.
The three different MIP formulations that we discuss differ in the way how
they model~$\rc_\varepsilon(X,Y)$.
The first model uses only polynomially many variables and
inequalities, the second model needs exponentially many inequalities while
the number of variables is still polynomial, and the third model requires
exponentially many variables but only polynomially many inequalities.
For this reason, we refer to these three models as the compact, the cutting
plane, and the column generation model, respectively.
In preliminary experiments with our code, we have already used the compact
and column generation model~\cite{AverkovEtAl2021}.
Nevertheless, we provide the full details of these models to make the article
self-contained and to be able to explain the model enhancements.
For the sake of convenience, we assume for the remainder of this article
that~$X$ and~$Y$ are disjoint.
This is without loss of generality, because we can replace~$Y$ by~$Y
\setminus X$, which does not change the value of~$\rc_\varepsilon(X,Y)$.
We also refer to~$X$ as the set of \emph{feasible} points, whereas the
points in~$Y$ are called \emph{infeasible}.

\subsection{Compact Model}

Observe that lattice-convex sets are exactly those subsets of~$\Z^d$ that admit a relaxation.
In~\cite{averkovschymura2020complexity}, a mixed-integer programming
formulation has been proposed to check whether a finite lattice-convex
set~$X$ admits a relaxation with~$k$ inequalities, and we have explained
in~\cite{AverkovEtAl2021} how to adapt the model to be able to
compute~$\rc_\varepsilon(X,Y)$.

Given an upper bound~$k$ on the number of inequalities needed to separate~$X$
and~$Y$, the model's idea is to introduce variables~$a_{ij}$
and~$b_i$, $(i,j) \in [k] \times [d]$, to model the~$k$ potential
inequalities needed in a relaxation.
Moreover, for each~$y \in Y$ and~$i \in [k]$, a binary variable~$s_{yi}$ is
introduced that indicates whether the~$i$-th inequality is violated by~$y$;
additional binary variables~$u_i$, $i \in [k]$, indicate whether the~$i$-th inequality
is needed in a relaxation.
Using a big-M term with~$M \geq d(\rho_X + \rho_Y) + \varepsilon$, with~$\rho_X =
\max \{ \| x\|_\infty \st x \in X\}$ and~$\rho_Y = \max \{ \| y\|_\infty
\st y \in Y\}$, the mixed-integer programming formulation for~$\rc_\varepsilon(X,Y)$ is as follows:
\begin{subequations}
  \label{eq:compactModel}
  \begin{align}
    \min \sum_{i = 1}^k u_i &&&\\
    \sum_{j = 1}^d a_{ij} x_j &\leq b_i, && x \in X,\; i \in [k],
    \label{eq:compactModelFeas}\\
    \sum_{i = 1}^k s_{yi} &\geq 1, && y \in Y,
    \label{eq:compactModelCovering}\\
    \sum_{j = 1}^d a_{ij} y_j &\geq b_i + \varepsilon - M(1 - s_{yi}),
                              && y \in Y,\; i \in [k],
    \label{eq:compactModelInf}\\
    s_{yi} &\leq u_i, && y \in Y,\; i \in [k],
    \label{eq:compactModelLink}\\
    -1 \leq a_{ij} &\leq 1, && (i,j) \in [k] \times [d],
    \label{eq:compactModelScaleA}\\
    -d \rho_X \leq b_i &\leq d \rho_X, && i \in [k],
    \label{eq:compactModelScaleB}\\
    s_{yi},\; u_i &\in \bin{}, && y \in Y,\; i \in [k].
    \label{eq:compactModelBin}
  \end{align}
\end{subequations}
Inequalities~\eqref{eq:compactModelFeas} ensure that the~$k$ inequalities
are valid for~$X$ and Inequalities~\eqref{eq:compactModelCovering}
guarantee that each~$y \in Y$ is cut off by at least one inequality.
If an inequality is selected to separate~$y \in Y$ and~$X$,
Inequalities~\eqref{eq:compactModelInf} ensure that this is consistent with
the~$k$ inequalities defined by the model.
Finally, Inequalities~\eqref{eq:compactModelLink} ensure that~$u_i$ is~1 if
inequality~$i \in [k]$ separates an infeasible point, whereas
Inequalities~\eqref{eq:compactModelScaleA}
and~\eqref{eq:compactModelScaleB} scale the~$k$ inequalities without loss
of generality.
For details on correctness, we refer the reader
to~\cite[Sect.~4.2]{averkovschymura2020complexity}.

\subsection{Cutting Plane Model}

To be able to find~$\rc_\varepsilon(X,Y)$, Model~\eqref{eq:compactModel}
introduces two classes of variables: variables~$u$ and~$s$ model which
inequalities are used and subsets of~$Y$ that are
separated by the selected inequalities, respectively, whereas variables~$a$
and~$b$ guarantee that the subsets defined by~$s$ can be cut by valid
inequalities for~$X$.
The problem of computing~$\rc_\varepsilon(X,Y)$ can thus be interpreted as
a two stage problem, where the first stage selects a set of subsets of~$Y$
and the second stage checks whether the selected subsets correspond to
feasible cut patterns.
Since the first stage variables are binary and the second stage problem is
a feasibility problem, logic-based Benders decomposition can be used to
compute~$\rc_\varepsilon(X,Y)$, see~\cite{Hooker2000}.
While classical Benders decomposition requires the subproblem to be a
linear programming problem, logic-based Benders decomposition allows the
subproblem to be an arbitrary optimization problem.

Let~$\mathcal{C} = \{ C \subseteq Y \st C \text{ and } X \text{ are not
  linearly $\varepsilon$-separable}\}$.
We refer to~$\mathcal{C}$ as the \emph{conflict set}.
For all~$(C, i) \in \mathcal{C} \times [k]$, the \emph{conflict inequality}
${\sum_{y \in C} s_{yi} \leq \card{C} - 1}$ models that not all points in~$C$
can be cut by an inequality valid for~$X$.
Consequently,
\begin{subequations}
  \label{eq:conflictModel}
  \begin{align}
    \min \sum_{i = 1}^k u_i &&&\\
    \sum_{i = 1}^k s_{yi} &\geq 1, && y \in Y,
    \label{eq:conflictModelCovering}\\
    \sum_{y \in C} s_{yi} &\leq \card{C} - 1, && C \in \mathcal{C},\; i \in [k],
    \label{eq:conflictModelConflict}\\
    s_{yi} &\leq u_i, && y \in Y,\; i \in [k],
    \label{eq:conflictModelLink}\\
    s_{yi},\; u_i &\in \bin{}, && y \in Y,\; i \in [k].
  \end{align}
\end{subequations}
is an alternative model for computing~$\rc_{\varepsilon}(X,Y)$.

\subsection{Column Generation Model}

Let~$\I = \{ I \subseteq Y \st I \text{ and } X \text{ are linearly
  $\varepsilon$-separable}\}$.
Then, $\rc_\varepsilon(X,Y)$ is the smallest number~$\ell$ of sets~$I_1,
\dots, I_\ell \in \I$ such that~$Y = \bigcup_{i = 1}^\ell I_i$.
Thus, instead of using the matrix~$s \in \bin{Y \times [k]}$ to encode which
inequality cuts which points from~$Y$, we can introduce for every~$I \in \I$ a binary variable~$z_I \in \bin{}$ that encodes whether an inequality separates~$I$ or not:
\begin{subequations}
  \label{eq:CGmodel}
  \begin{align}
    \min \sum_{I \in \I} z_I &&&\\
    \sum_{I \in I_y} z_I &\geq 1, && y \in Y,\label{eq:CGmodelCovering}\\
    z &\in \Z_+^{\I},&&
  \end{align}
\end{subequations}
where~$I_y = \{ I \in \I \st y \in I\}$.
\begin{remark}
  In contrast to Model~\eqref{eq:compactModel},
  Models~\eqref{eq:conflictModel} and~\eqref{eq:CGmodel} do not directly
  provide an~$\varepsilon$-relaxation of~$X$ w.r.t.~$Y$.
  To find such a relaxation, $\rc_\varepsilon(X,Y)$ many linear programs need to
  be solved in a post-processing step.
\end{remark}

\section{Comparison of Basic Models}
\label{sec:modelComparison}

While the compact model~\eqref{eq:compactModel} can be immediately handed to an MIP solver due to the relatively small number of variables and
constraints, the cutting plane model~\eqref{eq:conflictModel} and column
generation model~\eqref{eq:CGmodel} require to implement separation and
pricing routines, respectively.
At least for the column generation model, this additional computational
effort comes with the benefit of a stronger LP relaxation in comparison
with the compact model.
To make this precise, we denote by~$\vcompact$, $\vcut$, and~$\vCG$ the optimal
LP relaxation value of the compact, cutting plane, and column generation
model, respectively.
\begin{proposition}
  \label{prop:qualityCut}
  Let~$X \subseteq \Z^d$ be finite and lattice-convex, let~$Y \subseteq
  \Z^d \setminus X$ be finite, let~$\varepsilon > 0$ such
  that~$\rc_\varepsilon(X,Y)$ exists, and suppose both
  Models~\eqref{eq:compactModel} and~\eqref{eq:conflictModel} are
  feasible.
  \begin{enumerate}
  \item Then, $\vcompact \geq \vcut = 1$.

  \item\label{item2} Moreover, if~$\varepsilon \leq (d-1)(\rho_X + \rho_Y)$, then~$\vcompact = 1$.
  \end{enumerate}
\end{proposition}
Note that~\ref{item2} is a technical assumption that is almost always
satisfied in practice, e.g., to approximate~$\rc(X,Y)$
by~$\rc_\varepsilon(X,Y)$, one selects~$\varepsilon < 1 \leq (d-1)(\rho_X +
\rho_Y)$.
Thus, $\vcut = \vcompact = 1$ in all relevant cases.
\begin{proof}
  First we show~$\vcut \geq 1$ and~$\vcompact \geq 1$.
  Observe that we get for every (partial) feasible solution~$(s,u)$
  and every~$\bar{y} \in Y$ the estimation
  \[
    \sum_{i = 1}^k u_i
    \geq
    \sum_{i = 1}^k \max\{s_{yi} \st y \in Y\}
    \geq
    \sum_{i = 1}^k s_{\bar{y}i}
    \geq
    1,
  \]
  where~$k$ is the upper bound used in Model~\eqref{eq:compactModel}
  of~\eqref{eq:conflictModel}.
  Hence, $\vcut \geq 1$ and $\vcompact \geq 1$.
  If the upper bound~$k = 1$, we thus have necessarily~$\vcut = 1$.
  If~$k \geq 2$, we construct a feasible solution
  for~\eqref{eq:conflictModel} with objective value~1 by assigning
  all variables value~0 except for~$s_{yi}$, $(y,i) \in Y \times [2]$,
  $u_1$, and~$u_2$, which get value~$\frac{1}{2}$.
  Indeed, the left-hand side of each conflict inequality evaluates to~$\frac{\card{C}}{2}$,
  while the right-hand side is~$\card{C} - 1$.
  Thus, because~$\card{C} \geq 2$ for any conflict as~$X$ is
  lattice-convex, the find~$\frac{\card{C}}{2} \leq \card{C} - 1$, i.e.,
  all conflict inequalities are satisfied.
  Since the remaining inequalities hold trivially, $1 \geq \vcut$ follows.
  Consequently, $\vcompact \geq 1 \geq \vcut \geq 1$.

  For the second statement, we assume~$k \geq 2$, because
  otherwise~${\vcompact = 1}$ follows as above.
  We define a feasible solution with objective value~1 of
  Model~\eqref{eq:compactModel} by assigning all variables value~0 except
  for
  \begin{itemize}
  \item $u_1 = s_{y1} = \frac{\varepsilon}{M}$ for all~$y \in Y$;
  \item $u_2 = s_{y2} = 1 - \frac{\varepsilon}{M}$ for all~$y \in Y$;
  \item $a_{11} = 1$ and~$b_1 = \rho_X$.
  \end{itemize}
  The inequalities~$\sprod{a_{i \cdot}}{x} \leq b_i$ defined this way are
  either~$0 \leq 0$ or~$x_1 \leq \rho_X$, which are valid for~$X$.
  Moreover, the Inequalities~\eqref{eq:compactModelInf} are satisfied,
  because for~$i = 1$ and every~$y \in Y$, we have
  \[
    \sum_{j = 1}^d a_{1j}y_j - b_1
    =
    y_1 - \rho_X
    \geq
    -\rho_Y - \rho_X
    \geq
    -d(\rho_X + \rho_Y) + \varepsilon
    \geq
    \varepsilon - M(1 - s_{y1}),
  \]
  and for the remaining~$i \geq 2$, we get
  $
  \varepsilon - M(1 - s_{y2})
  =
  0.
  $
  Since one can easily check that the remaining inequalities
  of~\eqref{eq:compactModel} are also satisfied, $\vcompact \leq 1$
  follows, concluding the proof using the first part of the assertion.
\end{proof}
The value of the LP relaxations thus does not indicate whether the compact
or cutting plane model performs better in practice.
An advantage of the latter is that the conflict inequalities encode a
hypergraph coloring problem, which is a structure appearing frequently in
practice.
Hence, there might be a chance that a solver can exploit this structure if
sufficiently many inequalities have been separated.
The compact model, however, might have the advantage that the $a$-
and~$b$-variables guide the solver in the right direction when branching
on~$s$- or~$u$-variables, because feasibility is already encoded in the
model and does not need to be added to the model by separating cutting
planes.
\begin{proposition}
  \label{prop:qualityCG}
  Let~$X \subseteq \Z^d$ be finite and lattice-convex, let~$Y \subseteq
  \Z^d \setminus X$ be finite, let~$\varepsilon > 0$ be such
  that~$\rc_\varepsilon(X,Y)$ exists, and suppose both
  Models~\eqref{eq:compactModel} and~\eqref{eq:conflictModel} are
  feasible.
  Let~$k$ be the number of inequalities encoded in
  Model~\eqref{eq:compactModel}.
  \begin{enumerate}
  \item If there exists an optimal solution of the LP relaxation
    of~\eqref{eq:CGmodel} that assigns at most~$k$ variables a positive
    value, then $\vCG \geq \vcompact \geq \vcut = 1$.

  \item We have $\vCG \geq H(X,Y)$, and this can be strict.
  \end{enumerate}
\end{proposition}
\begin{proof}
  To show~$\vCG \geq \vcompact$, recall that for each~$I \in \I$ there
  exists an inequality~$\sprod{a(I)}{x} \leq b(I) + \varepsilon$
  separating~$I$ and~$X$.
  Due to rescaling, we may assume that~$a(I) \in [-1,1]^d$ and~$b(I) \in [-d
  \rho_X, d \rho_X]$.

  If we are given a solution~$z \in \R_+^\I$ of~\eqref{eq:CGmodel} with at
  most~$k$ non-zero entries, we define a solution
  of the LP relaxation of~\eqref{eq:compactModel} with the same objective value as
  follows.
  Let~$I_1, \dots, I_{\ell} \in \I$ be the indices of non-zero entries
  in~$z$.
  For each~$i \in [\ell]$ and~$y \in Y$, define
  \[
    s_{yi} =
    \begin{cases}
      z_{I_i}, & \text{if } y \in I_i,\\
      0, & \text{otherwise},
    \end{cases}
    \qquad
    \text{and}
    \qquad
    u_i = z_{I_i}.
  \]
  For~$i \in \{\ell + 1, \dots, k\}$ and~$y \in Y$, we define~$s_{yi} = 0$
  and~$u_i = 0$.
  Finally, let~$a_{ij} = a(I_i)_j$ and~$b_i = b(I_i)$ for~$(i,j) \in
  [\ell] \times [d]$.
  For~$i \in \{\ell + 1, \dots, k\}$, define~$a_{ij} = 0$ and~$b_i = 1$.
  Indeed, this solution adheres to~\eqref{eq:compactModelFeas}
  since~$(a,b)$ defines valid inequalities, and
  also~\eqref{eq:compactModelLink}--\eqref{eq:compactModelScaleB} hold
  trivially.
  By definition, $s$ and~$u$ also satisfy the box constraints corresponding
  to~\eqref{eq:compactModelBin}.
  To see that~\eqref{eq:compactModelCovering} holds, note that for each~$y
  \in Y$,
  \[
    \sum_{i = 1}^k s_{yi}
    =
    \sum_{i \in [\ell]\colon y \in I_i} z_{I_i}
    \overset{\eqref{eq:CGmodelCovering}}{\geq}
    1,
  \]
  since~$z$ is feasible for the LP relaxation of~\eqref{eq:CGmodel}.
  For the last constraint~\eqref{eq:compactModelInf}, note that the
  constraint is trivially satisfied if~$s_{yi} = 0$.
  If~$s_{yi} > 0$, then~$\sprod{a_{i \cdot}}{x} \leq b_i$ corresponds to an
  inequality separating~$X$ and~$y$, which finally shows that the newly
  defined solution is feasible for the LP relaxation
  of~\eqref{eq:compactModel}.
  To conclude, note that~$\sum_{i = 1}^k u_i = \sum_{i = 1}^\ell z_{I_i}$.
  Hence, $\vCG \geq \vcompact$ and the remaining estimations hold by
  Proposition~\ref{prop:qualityCut}.

  For the second part,
  let~$H \subseteq Y$ be a hiding set for~$X$ and let~$z \in \R_+^\I$ be an
  optimal solution of the LP relaxation of~\eqref{eq:CGmodel}.
  Then, for distinct~$y_1, y_2 \in H$, we have~$I_{y_1} \cap I_{y_2} =
  \emptyset$.
  Consequently, we can estimate
  \[
    \vCG
    =
    \sum_{I \in \I} z_I
    \geq
    \sum_{y \in H} \sum_{I \in \I_y} z_I
    \overset{\eqref{eq:CGmodelCovering}}{\geq}
    \card{H},
  \]
  which shows~$\vCG \geq H(X,Y)$.

  To see that the inequality can be strict, consider~$X = \bin{2}$ and
  let~$Y$ be all infeasible points in~$\Z^2$ with~$\ell_\infty$-distance~1
  from~$X$.
  One can readily verify that a maximum hiding set for~$X$ has size~2,
  while the LP relaxation of~\eqref{eq:CGmodel} has value~$\frac{8}{3}$.
\end{proof}
If~$Y$ contains a hiding set of size at least~2, the column
generation model is thus strictly stronger than the compact and cutting plane
model.
In particular, the gap between~$\vCG$ and~$\vcut$ (and~$\vcompact$) can be
arbitrarily large: if~$d=2$ and $Y = \obs(X)$, there is always a
hiding set of size~$\rc(X,Y) - 1$, see~\cite[Thm.~23]{AverkovEtAl2021}.

\section{Enhancements of Basic Models and Algorithmic Aspects}
\label{sec:modelEnhancements}

In their basic versions, the compact and cutting plane model are rather
difficult to solve for a standard MIP solver, e.g., because not enough
structural properties of~$\rc_\varepsilon(X,Y)$ are encoded in the models
that are helpful for a solver.
Moreover, the cutting plane and column generation model require to solve a
separation and pricing problem, respectively, to be used in practice.
In this section, we discuss these aspects and suggest model improvements.

\subsection{Incorporation of Structural Properties}

In the following, we describe cutting planes, propagation algorithms, and
techniques to handle symmetries and redundancies in the compact and cutting plane
model.

\paragraph{\textbf{Cutting Planes}}
In both the compact and cutting plane model, variable~$s_{yi}$ encodes whether
a point~$y \in Y$ is separated by inequality~$i \in [k]$.
To strengthen the compact model and the initial LP without separated
inequalities in the cutting plane model, we can add inequalities that rule
out combinations of points from~$Y$ that cannot be separated
simultaneously.

For any hiding set~$H \subseteq Y$, the \emph{hiding set cut}
\begin{align*}
  \sum_{y \in H} s_{yi} &\leq 1, & i \in [k]
\end{align*}
encodes that each inequality~$i \in [k]$ can separate at most one element
from a hiding set.
Although these cuts are the stronger the bigger the underlying hiding set, we
add these inequalities just for hiding sets of size~2.
The reason for this is that such hiding sets can be found easily by iterating
over all pairs~$(y_1, y_2)$ of distinct points in~$Y$ and checking whether the line
segment~$\conv(\{y_1, y_2\})$ intersects~$\conv(X)$ non-trivially.
In our implementation, we insert the expression~$\lambda y_1 +
(1-\lambda)y_2$ in each facet defining inequality of~$\conv(X)$ to derive
bounds on the parameter~$\lambda$.
Then, the final bounds on~$\lambda$ are within~$[0,1]$ if and only
if~$\{y_1,y_2\}$ is a hiding set.

For hiding sets of arbitrary cardinality, the task is more
difficult, because there might exist exponentially many
hiding sets.
Thus, we are relying on a separation routine for hiding set cuts.
The separation problem for hiding set cuts, however, is at least as
difficult as finding a maximum hiding set for~$X$, and the complexity of
the latter is open.

\paragraph{\textbf{Propagation}}
Suppose we are solving the compact and cutting plane model using
branch-and-bound.
At each node of the branch-and-bound tree, there might exist some binary
variables that are fixed to~0 or~1, e.g., by branching decisions.
The aim of propagation is to find further variable fixings based on the
already existing ones.

Our first propagation algorithm is based on the following observation.
\begin{observation}
  Suppose some~$s$-variables have been fixed and let~$i \in [k]$.
  Then, $F_i \define \{y \in Y \st s_{yi} = 1\}$ can be separated from~$X$
  if and only if~$F_i' \define Y \cap \conv(F_i)$ can be separated from~$X$.
\end{observation}
The \emph{convexity propagation algorithm} computes the sets~$F_i'$, $i \in
[k]$, and fixes~$s_{yi}$ to~1 for all~$y \in F_i'$.
If there is~$y' \in F_i'$ such that~$s_{y'i}$ is already fixed to~0, then
the algorithm prunes the node of the branch-and-bound tree.
This is indeed a valid operation, because
Inequalities~\eqref{eq:compactModelCovering}
and~\eqref{eq:conflictModelCovering} allow each point~$y \in Y$ to be
separated by several inequalities.

The second propagation algorithm exploits that~$F_i \cap \conv(X)$ needs to
be empty in each feasible solution.
The \emph{intersection propagation algorithm} thus iterates over all~$y \in
Y \setminus F_i$ and checks whether~$\conv(F_i \cup \{y\}) \cap \conv(X)
\neq \emptyset$.
If the check evaluates positively, $s_{yi}$ is fixed to~0.

Comparing both propagation algorithms, the convexity propagator requires to
compute only a single convex hull per set~$F_i$, whereas the intersection
propagator needs to compute~$O(\card{Y})$ convex hulls per set~$F_i$, which
can be rather expensive.
In our experiments, we will investigate whether the additional effort pays
off in reducing the running time drastically.
To avoid computing unnecessary convex hulls, we call both propagation
algorithms in our implementation only if the branching decision at the
parent node is based on a variable~$s_{yi}$, and in this case only for this
particular inequality index~$i$ and no further~$i' \in [k] \setminus \{i\}$.

\paragraph{\textbf{Symmetry Handling}}
It is well-known that the presence of symmetries slows down MIP
solvers, because symmetric solutions are found repeatedly during the
solving process leading to an exploration of unnecessary parts of the
search space.
In a solution of the compact and cutting plane model, e.g., we can permute
the inequality labels~$i \in [k]$ without changing the structure of the
solution.
For this reason, one can enforce that only one representative solution per
set of equivalent solutions is computed without losing optimal solutions.

One way of handling symmetric relabelings of inequalities is to require
that the columns of the matrix~$s \in \bin{Y \times [k]}$ are sorted
lexicographically non-increasingly.
To enforce sorted columns, we use a separation routine for orbisack minimal
cover inequalities as suggested in~\cite{HojnyPfetsch2019} and the
propagation algorithm orbitopal fixing by Bendotti et
al.~\cite{BendottiEtAl2021}.
Both algorithms' running time is in~$O(|Y| \cdot k)$.
Moreover, sorting the columns of~$s$ implies that we can also require
the~$u$-variables to be sorted, i.e., the first~$\rc_\varepsilon(X,Y)$
inequalities are
the inequalities defining an~$\varepsilon$-relaxation, which can be
enforced by adding
\begin{align}
  \label{eq:sortU}
  u_i &\geq u_{i+1}, && i \in [k-1],
\end{align}
to the problem.

Besides the symmetries of relabeling inequalities, we might also be able to
relabel points in~$Y$ without changing the structure of the problem.
This is the case if we find a permutation~$\pi$ of~$[d]$ such that~$\pi(X)
= X$ and~$\pi(Y) = Y$, where for a set~$T \subseteq \R^d$ we define~$\pi(T)
= \{\pi(t) \st t \in T\}$ and~$\pi(t) = (t_{\pi^{-1}(1)}, \dots,
t_{\pi^{-1}(d)})$.
The permutation~$\pi$ gives rise to a permutation~$\phi$ of~$Y$ and~$\psi$
of~$X$, where~$\phi(y) \define \pi(y)$ and~$\psi(x) \define \pi(x)$.
\begin{lemma}
  \label{lem:sortRows}
  Let~$(s,u)$ be a (partial) solution of Model~\eqref{eq:compactModel}
  or~\eqref{eq:conflictModel} for~$\rc_\varepsilon(X,Y)$.
  If there exists a permutation~$\pi$ of~$[d]$ such that~$\pi(X) = X$
  and~$\pi(Y) = Y$, then also~$(s',u)$ is a (partial) solution, where~$s'$
  arises from~$s$ by reordering the rows of~$s$ according to~$\phi$.
\end{lemma}
\begin{proof}
  Suppose~$(s,u)$ is a solution of Model~\eqref{eq:conflictModel}.
  Then, $(s,u)$ can be extended to a solution of
  Model~\eqref{eq:compactModel}, i.e., there exist~$k$
  inequalities~$\sum_{j = 1}^d a_{ij}x_j \leq b_i$, ${i \in [k]}$, such that
  the~$i$-th inequality separates the points in~$F_i = \{y \in Y \st s_{yi}
  = 1\}$ from~$X$.
  If we apply permutation~$\pi$ to~$X$ and~$Y$, we do not change the
  structure of the problem, that is, ~$\sum_{j = 1}^d a_{ij}\pi(x)_j \leq
  b_i$, $i \in [k]$, defines also a relaxation of~$X$ w.r.t.~$Y$.
  Thus, if the original~$i$-th inequality separated point~$y \in Y$, the
  permuted inequality separates~$\phi(y)$.
  Consequently, if we define~$s'$ by relabeling the rows of~$s$ according to~$\phi$,
  $(\pi(a), b, s', u)$ is a solution of Model~\eqref{eq:compactModel}
  and thus~$(s', u)$ is a solution of Model~\eqref{eq:conflictModel}.
\end{proof}
If~$\Pi = \{ \pi \in S_d \st \pi(X) = X,\; \pi(Y) = Y\}$ and~$\Phi$ is the
group containing all~$\phi$ associated with the permutations~$\pi \in \Pi$,
Lemma~\ref{lem:sortRows} tells us that we can also force the rows of~$s$ to
be sorted lexicographically non-increasingly w.r.t.\ permutations from~$\Phi$.
In our implementation, we compute a set~$\Gamma$ of generators of the
group~$\Phi$ and enforce for each~$\gamma \in \Gamma$ that matrix~$s$ is
lexicographically not smaller than the reordering of~$s$ w.r.t.~$\gamma$.
We enforce this property by separating minimal cover
inequalities for symresacks and a propagation algorithm, see~\cite{HojnyPfetsch2019}.
Both run in~$O(k)$ time per~$\gamma \in \Gamma$.

To detect the symmetries~$\Phi$, we construct a colored bipartite graph~$G = (V,E)$.
The left side of the bipartition is given by~$X \cup Y$ and the right
side is defined as~$R = \{ (v,j) \in \Z \times [d] \st \text{there is } z \in X \cup Y
\text{ with } z_j = v\}$.
There is an edge between~$z \in X \cup Y$ and~$(v,j) \in R$ if and only
if~$z_j = v$.
Moreover, each node gets a color uniquely determining its type: all nodes in~$X$ are
colored equally with color ``$X$'', all nodes in~$Y$ are colored equally by
color~``$Y$'', and node~$(v,j) \in R$ is colored by color~``$v$''.
Then, the restriction of every automorphism~$\sigma$ of~$G$ to~$R$ corresponds to a
permutation in~$\Pi$, and thus, restricting~$\sigma$ to~$Y$ is a
permutation in~$\Phi$.

Note that the graph~$G$ defined above might not allow to detect symmetries
if a symmetric arrangement of~$X$ and~$Y$ is translated asymmetrically.
For example, if~$X = t + \Delta_2$, $Y = t + (\Delta_2 + \lozenge_2) \setminus \Delta_2$, and~$t = \binom{1}{2}$,
then there is no permutation keeping~$X$ invariant.
For this reason, we use in the construction of~$G$ relative coordinates.
That is, for each coordinate~$j \in [d]$, we compute~$\mu_j = \min_{z \in X \cup Y}
z_j$ and translate~$X \cup Y$ by~$-\mu$ before building~$G$.

\medskip
Another way of handling symmetries for the compact
model~\eqref{eq:compactModel} is to handle symmetries of the
inequalities~$\sum_{j = 1}^d a_{ij}x_j \leq b_i$ defined in the model.
We can reorder the inequalities~$\sum_{j = 1}^d a_{ij} x_j \leq b_i$,
$i \in [k]$ that are (not) used in the relaxation, to obtain another
solution with the same objective value.
To handle these symmetries, we can add the inequalities
\begin{align}
  a_{i1} & \geq a_{(i+1)1} - 2(u_i - u_{i+1}), && i \in [k-1].\label{eq:sortA}
\end{align}
Inequalities~\eqref{eq:sortA} sort the inequalities (not) present in a
relaxation by their first coefficient.
The inequalities are compatible with Inequalities~\eqref{eq:sortU}, but not
necessarily with the lexicographic ordering constraints.
The latter is the case because cutting the point~$y \in Y$ associated with
the first row of matrix~$s$ might require a very small first coefficient in
any separating inequality, whereas other points might require a very large
first coefficient.
In our experiments, we will investigate which symmetry handling method
works best for the compact and cutting plane model.

Finally, additional redundancies in Model~\eqref{eq:compactModel} can be
handled by enforcing that~$\sum_{j = 1}^d a_{ij} x_j \leq b_i$ becomes the
trivial inequality~$\sprod{0}{x} \leq d \rho_X$ if it is not used in a
relaxation of~$X$ (i.e., \mbox{$u_i = 0$}).
This removes infinitely many equivalent solutions from the search space,
and can be modeled by replacing~\eqref{eq:compactModelScaleA} by
\begin{align*}
  -u_i \leq a_{ij} \leq u_i, && (i,j) \in [k] \times [d],
\end{align*}
and the lower bound constraint in~\eqref{eq:compactModelScaleB} by
\begin{align*}
  d \rho_X \leq b_i + 2 d \rho_X u_i, && i \in [k].
\end{align*}
This method is compatible with both the lexicographic symmetry handling
approach and Inequalities~\eqref{eq:sortA}.

\subsection{Algorithmic Aspects of the Cutting Plane Model}

To be able to deal with the exponentially many conflict
inequalities~\eqref{eq:conflictModelConflict} in the cutting plane
Model~\eqref{eq:conflictModel}, we are relying on a separation routine.
We start by discussing the case that the point~$s^\star$ to be separated is
contained in~$\bin{Y \times [k]}$, i.e., for each of the~$k$ inequalities
we already know which points it is supposed to separate.

To check whether~$s^\star \in \bin{Y \times [k]}$ satisfies all conflict inequalities, we can
compute for each~$i \in [k]$ the set~$F_i = \{ y \in Y \st s^\star_{yi} =
1\}$, and build a linear program similar to Model~\eqref{eq:compactModel}
that decides whether~$X$ and~$F_i$ are $\varepsilon$-separable.
If the answer is yes, we know~$s^\star$ is feasible.
Otherwise, we have found a violated conflict inequality, namely $\sum_{y
  \in F_i} s_{yi} \leq \card{F_i} - 1$.
Of course, this inequality will be rather weak in practice, because it
excludes only the single assignment~$F_i$.

One way to strengthen the inequality is to search for a minimum cardinality
subset~$F_{\min}$ of~$F_i$, which cannot be separated from~$X$.
The corresponding inequality~$\sum_{y \in F_{\min}} s_{yi} \leq \card{F_{\min}} - 1$ then
does not only cut off~$s^\star$, but every solution that assigns
inequality~$i$ all points from~$F_{\min}$.
However, we do not expect that~$F_{\min}$ can be computed efficiently,
because detecting a minimum cardinality set of inequalities whose removal
leads to a feasible LP is NP-hard, see Sankaran~\cite{Sankaran1993}.
Instead, we compute a minimal cardinality subset~$F \subseteq F_i$ by
initializing~$F = \emptyset$, adding points~$y \in F_i$ to~$F$ until~$F$
and~$X$ are no longer separable, and then iterating over all points~$y'$ in~$F$
and checking whether their removal leads to a separable set.
In the latter case, we keep~$y'$ in~$F$; otherwise, we remove it.
Although this procedure is costly as it requires to
solve~$\Theta(\card{F_i})$ LPs to find~$F$, preliminary experiments
revealed that the running time of the cutting plane model can be reduced
drastically when using the sparsified inequalities.

Since we expect the separation problem of~\eqref{eq:conflictModelConflict}
to be difficult even for integer points, we only heuristically separate
non-integral points~$s^\star \in [0,1]^{Y \times [k]}$ in our implementation.
To this end, for each~$i \in [k]$, we again initialize an empty set~$F$ and
iteratively add~$y \in Y$ in non-increasing order w.r.t.~$s^\star_{yi}$
until~$F \in \mathcal{C}$ and~$s^\star$ violates the inequality (or we know
that such an inequality cannot be violated).

\subsubsection{Algorithmic Aspects of the Column Generation Model}

In contrast to the compact model~\eqref{eq:compactModel}, the number of variables
in~\eqref{eq:CGmodel} grows exponentially in~$\card{Y}$, which makes it
already challenging to solve the LP relaxation of~\eqref{eq:CGmodel}.
In our implementation, we thus use a branch-and-price procedure for
solving~\eqref{eq:CGmodel}, i.e., we use a branch-and-bound procedure in
which each LP relaxation is solved by column generation.
In the following, we discuss the different components of the
branch-and-price procedure.

\paragraph{\textbf{Solving the Root Relaxation}}
At the root node of the branch-and-bound tree, we are given a subset~$\I'$ of
all possible variables in~$\I$ and solve the LP relaxation
of~\eqref{eq:CGmodel} restricted to the variables in~$\I'$.
To check whether the solution obtained for the variables in~$\I'$ is indeed
an optimal solution of the LP relaxation, we need to solve the pricing
problem, i.e., to check whether all
variables in~$\I$ have non-negative reduced costs.
Since the pricing problem is equivalent to the separation problem for the
dual, we determine the dual of the root node LP relaxation
of~\eqref{eq:CGmodel}, which is given by
\begin{subequations}
  \begin{align}
    \label{eq:dualCG}
    \max \sum_{y \in Y} \alpha_y &&&\\
    \sum_{y \in I} \alpha_y &\leq 1, && I \in \I,\\
    \alpha_y &\geq 0, && y \in Y.
  \end{align}
\end{subequations}
The pricing problem at the root node is thus to decide, for given dual
weights~$\alpha_y$, $y \in Y$, whether there exists a set~$I \in \I$
with~$\sum_{y \in I} \alpha_y > 1$.
Unfortunately, we cannot expect to solve this problem efficiently in
general.
\begin{proposition}
  Let~$X \subseteq \Z^d$ be finite and lattice-convex, let~$Y \subseteq
  \Z^d \setminus X$ be finite, and let~$\alpha_y \geq 0$ be a rational weight for~$y
  \in Y$.
  Then, the pricing problem for the LP relaxation of~\eqref{eq:CGmodel},
  i.e., deciding whether there exists~$I \in \I(X,Y)$ with
  $\sum_{y \in Y} \alpha_y > 1$, is NP-hard.
\end{proposition}
\begin{proof}
  Note that the pricing problem is equivalent to finding a set~$I \in
  \I(X,Y)$ that maximizes the value~$\sum_{y \in I} \alpha_y$.
  If all weights~$\alpha_y$, $y \in Y$, have the same value~$\alpha > 0$,
  the problem reduces to find a set~$I \in \I$ of maximum cardinality.
  The latter problem is NP-hard even if~$X$ consists of a single point, in
  which case it reduces to the open hemisphere problem, see Johnson \&
  Preparata~\cite{JohnsonPreparata1978}.
\end{proof}
To solve the pricing problem, we use a mixed-integer program that is a
variant of~\eqref{eq:compactModel} with~$k=1$.
The only difference is that instead of minimizing the number of needed
inequalities, we maximize the expression~$\sum_{y \in Y} \alpha_y s_{y1}$.
If this value is at most~1, we have found an optimal solution of the LP
relaxation.
Otherwise, we have found a variable~$z_{I}$ with negative
reduced cost, add~$I$ to~$\I'$, and iterate this procedure until all
reduced costs are non-negative.
In our implementation, we initialize the set~$\I'$ by
\begin{align*}
  \I' &= \big\{
    \{ y \in Y : \sprod{a}{y} > b \} : \sprod{a}{x} \leq b
    \text{ defines facet of}\,\conv(X)
  \big\} \cup
  \big\{
    \{y\} : y \in Y
  \big\}.
\end{align*}

\paragraph{\textbf{Branching Strategy}}
Let~$u$ be a node of the branch-and-bound tree and denote by~$z^u$ an
optimal solution of the LP relaxation at node~$u$.
A classical branching strategy is to select a variable~$z_I$ with~$z^u_I
\notin \Z$ and to create two child nodes~$u^0$ and~$u^1$
by enforcing~$z_I = 0$ in~$u^0$ and~$z_I = 1$ in~$u^1$.
While the branching decision~$z_I = 1$ has strong implications for computing~$\rc_{\varepsilon}(X,Y)$
(we basically fix an inequality used in the relaxation), branching~$z_I =
0$ only rules out one of the exponentially many choices in~$\I$ for a
separated set.

To obtain a more balanced branching rule, we use the branching
rule suggested by Ryan \& Foster~\cite{RyanFoster1981}.
We are looking for two distinct
variables~$z_I$ and~$z_J$ with~$z^u_I, z^u_J \notin \Z$ such that both the
intersection~$I \cap J$ and symmetric difference~$I \Delta J$ of~$I$
and~$J$ are non-empty.
Let~$y_1 \in I \cap J$ and~$y_2 \in I \Delta J$.
Then, two child nodes~$u^0$ and~$u^1$ of~$u$ are created as follows.
In~$u^0$, variables~$z_{I'}$ are fixed to~0 if~$I'$ contains both~$y_1$
and~$y_2$.
In~$u^1$, we fix~$z_{I'}$ to~0 if~$I'$ contains either~$y_1$ or~$y_2$.
That is, $u^0$ enforces~$y_1$ and~$y_2$ to be contained in different sets,
and~$u^1$ forces them to be contained in the same set~$I'$.
This branching rule obviously partitions the integer solutions feasible at
node~$u$.
To show its validity it is thus sufficient to show that for every
non-integral solution~$z^u$ the sets~$I$ and~$J$ exist.
\begin{lemma}
  Let~$z^u$ be a non-integral optimal solution of the LP relaxation
  of~\eqref{eq:CGmodel} at node~$u$ of the branch-and-bound tree.
  Then, there exist two distinct sets~$I, J \in \I'$ with~$z^u_I, z^u_J
  \notin \Z$ such that~$I \cap J \neq \emptyset$ and~$I \Delta J \neq
  \emptyset$.
\end{lemma}
\begin{proof}
  Let~$I \in \I'$ be such that~$z^u_{I} \notin \Z$.
  Then, $z^u_I \in (0,1)$, since~$z^u$ is an \emph{optimal} solution of the LP
  relaxation.
  Due to~\eqref{eq:CGmodelCovering}, for every~$y \in I$, there exists~$J^y
  \in \I' \setminus \{I\}$ with~$y \in J^y$ such that~$z^u_{J^y} > 0$.
  For at least one~$J^y$ we have~$z^u_{J^y} \in (0,1)$, because otherwise,
  we could improve the objective value of~$z^u$ by setting~$z^u_I$ to~0 and
  still satisfying all constraints.
  Such a set~$J^y$ together with~$I$ satisfy the properties in the
  statement of the lemma:
  Since~$y$ is contained in both~$I$ and~$J^y$, we have~$I \cap J^y \neq
  \emptyset$.
  Moreover, as~$I \neq J^y$, $I \Delta J^y \neq \emptyset$.
\end{proof}
In our implementation, we compute for each variable~$z^u_I$ its
fractionality $\theta(I) = \frac{1}{2} - \min\{z^u_I, 1 - z^u_I\}$.
Then, we select~$I$ and~$J$ such that~$\theta(I) + \theta(J)$ is maximized;
the branching candidates~$y_1 \in I \cap J$ and~$y_2 \in I \Delta J$ are
selected arbitrarily.

\paragraph{\textbf{Solving LP Relaxations in the Tree}}
To not re-generate variables that have been fixed to~0 by the branching
rule, we need to incorporate the branching decisions active at a node of
the branch-and-bound tree into the pricing problem.
This can easily be done by adding linear constraints to the root node
formulation of the pricing problem.
If a branching decision was that~$y_1$ and~$y_2$ shall be contained in
different sets, we add~$s_{y_1 1} + s_{y_2 1} \leq 1$ to the pricing problem.
The branching decision that~$y_1$ and~$y_2$ have to be contained in the
same set can be enforced by the constraint~$s_{y_1 1} = s_{y_2 1}$.

\section{Numerical Experiments}
\label{sec:experiments}

The aim of this section is to compare the practical performance of the
three models for computing~$\rc_\varepsilon(X,Y)$ as well as their
enhancements.
To this end, we have implemented all three models in C/C++ using
\texttt{SCIP~7.0.3} as modeling and branch-and-bound framework and
\texttt{SoPlex~5.0.2} to solve all LP relaxations.
All branching, propagation, separation, and pricing methods are implemented
using the corresponding plug-in types of \texttt{SCIP}.
Since we are not aware of an alternative separation routine for hiding set
cuts, we compute all hiding sets of size two in a straightforward fashion
before starting the branch-and-bound process.
During the solving process, we separate these inequalities if the
corresponding cuts are violated.
To handle symmetries via lexicographic orderings, we use \texttt{SCIP}'s internal
plug-ins \texttt{cons\_orbitope}, \texttt{cons\_orbisack}, and
\texttt{cons\_symresack} that implement the methods discussed in
Section~\ref{sec:modelEnhancements}; the branching and pricing plug-ins for
the column generation model strongly build up on the corresponding plug-ins
of the binpacking example provided in the SCIP Optimization Suite.
All convex hull computations have been carried out using \texttt{cdd~0.94m}~\cite{fukuda1997cdd}
and graph symmetries are detected using \texttt{bliss~0.73}~\cite{bliss}.

Our implementation is available online at
github\footnote{\url{https://github.com/christopherhojny/relaxation_complexity} (githash 4ffb6c0e was
  used for our experiments)}.

\paragraph{\textbf{Implementation Details}}
All models admit some degrees of freedom that we detail in the following.
Both the compact model and the cut model require an upper bound on the relaxation complexity.
In both models, we impose the trivial upper bound which is given by the
number of facets of~$\conv(X)$.
We also use the facet description to derive an initial solution for both
models.
In the column generation model, we need to select a subset of~$\I$ to
define initial variables.
We use the sets~$I \in \I$ that are defined by the facet defining
inequalities of~$\conv(X)$, i.e., the sets of points in~$Y$ that are
separated from~$X$ by the facet defining inequalities.
Moreover, we include the singleton sets~$\{y\}$, for~$y \in Y$, to make sure
that the LP relaxation remains feasible after branching.

\paragraph{\textbf{Settings}}
To encode the different settings that we have tested, we make use of the
following abbreviations:
\begin{itemize}[left= 24pt]
\item[hiding] Whether hiding set cuts are added (1) or not (0).
\item[sym.] Which symmetry method is used: none (0), simple (s), or
  advanced (a), where simple is~\eqref{eq:sortU} and~\eqref{eq:sortA}, and
  advanced uses~\eqref{eq:sortU} and additionally enforcing lexicographically maximal
  solutions based on symmetries of~$X$ and~$Y$.
\item[prop.] Whether the convexity propagator is used (1) or not (0).
\end{itemize}
Note that we do not report on results for the intersection propagation
algorithm.
This is because, in preliminary experiments, we have seen that its running time is very
high, in particular, because it needs to compute in each
iteration~$\cO(\card{Y})$ convex hulls.
As a result, we could hardly solve any instance, not even small ones.

\paragraph{\textbf{Test Sets}}
In our experiments, we have used three different test sets:
\begin{itemize}[left= 35pt]
\item[\basic]
  The sets~$X$ are the vertices of the 0/1 cube, the crosspolytope, or
  the standard simplex in dimensions~$d \in \{3,4,5\}$.
  For~$X \subseteq \Z^d$, the sets~$Y$ consist of all points in~$\Z^d
  \setminus X$ whose~$\ell_1$-distance to~$X$ is at most~$k$, where~\mbox{$1 \leq k \leq 10 - d$}.
  The reason for smaller distance in higher dimension is that the problems
  get considerably more difficult to solve with increasing~$k$.

\item[\dclosed] This test set consists of 99 full-dimensional subsets~$X$ of $\{0,1\}^5$ that correspond to down-closed subsets (or abstract simplicial complexes) of the Boolean lattice on~$5$ elements.
  The corresponding sets~$Y$ are the points in~$\Z^5 \setminus X$
  whose~$\ell_1$-distance to~$X$ is at most~$k \in \{1,2,3\}$.
  The sets~$X$ have been generated by the natural one-to-one correspondence between inclusion-maximal sets in a down-closed family and antichains in the Boolean lattice.
  
\item[\sboxes] The test set comprises~18 instances modeling 4-bit (12
  instances) and 5-bit (6 instances)
  S-boxes, which are certain non-sparse Boolean functions arising in
  symmetric-key cryptography.
  The derived sets~$X$ are contained in~$\bin{8}$ and~$\bin{10}$,
  respectively, and~$Y$ are the complementary binary points.
  These instances have also been used by Udovenko~\cite{Udovenko2021} who solved the full
  model~\eqref{eq:CGmodel}, i.e., without column generation.
\end{itemize}
The \basic instances feature various aspects that might be relevant for
computing~$\rc(X)$ via computing a series of values~$\rc_\varepsilon(X,Y)$
for different~$Y$ and~$\varepsilon$ according to~\cite{averkovschymura2020complexity}:
The cube is parity complete, thus there exists a small set~$Y$ such
that~$\rc(X) = \rc_\varepsilon(X,Y)$ (in fact, this set is~$\{-1,0,1,2\}^d
\setminus X$); the crosspolytope has an interior integer
point and thus there exists a (potentially large) finite set~$Y$
with~$\rc(X) = \rc_\varepsilon(X,Y)$; for the simplex~$\Delta_4$ in~$\R^4$, no finite set~$Y$
exists with $\rc(\Delta_4) = \rc(\Delta_4,Y)$; see~\cite{AverkovEtAl2021}.
That is, $\rc(\Delta_4, Y) < \rc(\Delta_4) \leq \rc_\Q(X,Y)$ for all finite sets~$Y \subseteq \Z^4$.

Since the standard simplex $\Delta_d$ is a down-closed subset of $\{0,1\}^d$, the small-sized \dclosed instances might be good candidates
for further examples~$X$ such that~$\rc_\varepsilon(X,Y) < \rc_\Q(X)$, for every finite set $Y \subseteq \Z^d$ and for $\varepsilon>0$ small enough.
Our aim for selecting these instances is thus to identify whether there are
potentially further candidates for sets~$X$ whose relaxation complexity
cannot be computed via finite sets~$Y$.

Finally, the \sboxes instances are used to investigate whether our techniques are
suited to compute~$\rc_\varepsilon(X,Y)$ also in higher dimensions.
This is relevant, among others, in the field of social choice or
symmetric cryptanalysis, where the aim is to find~$\rc(X, \bin{d})$ for
sets~$X \subseteq \bin{d}$.

\paragraph{\textbf{Computational Setup}}

All experiments have been run on a Linux cluster with Intel Xeon E5
\SI{3.5}{\GHz} quad core processors and~\SI{32}{\giga\byte} memory.
The code was executed using a single thread and the time limit for all
computations was~\SI{4}{\hour} per instance.

All mean numbers are reported in shifted geometric mean~$\prod_{i = 1}^n
(t_i + s)^{\frac{1}{n}} - s$ to reduce the impact of outliers.
For mean running times, a shift of~$s = 10$ is used; for nodes of the
branch-and-bound tree, we use~$s = 100$.
The value of~$\varepsilon$ in computing~$\rc_\varepsilon(\cdot, \cdot)$ is
set to~$0.001$.
The upper bound on the number of inequalities needed in the compact and
cutting plane model is given by the number of facets of~$\conv(X)$.
We also provide an initial primal solution corresponding to a facet
description of~$\conv(X)$.

\subsection{Results for Test Set \basic}

Due to our choice of the sets~$X$ and~$Y$, the basic test set comprises~18 cube,
crosspolytope, and simplex instances, respectively.
Table~\ref{tab:compact} shows the results for the compact model.
For the plain compact model, we observe that \texttt{SCIP} can already
solve quite some instances, but, in comparison to the enhanced variants,
the running times are rather high.
Checking each of the enhancements separately, handling symmetries is most
important to reduce running time and to increase the number of instances
solvable within the time limit.
Interestingly, handling symmetries on the~$a$-variables modeling the
inequalities in a relaxation performs better than handling the symmetries
of the points to separate.
Adding hiding set cuts to the problem formulation is also beneficial,
whereas the convexity propagator seems to harm the solving process in
particular for cube instances.
The worse performance for enabled propagation cannot be explained on the
running time of the propagator:
For cube instances, e.g., the maximum running time per instance of the
propagator was~\SI{27}{\s}, which is much smaller than the increase of mean
running time.
Thus, it seems that the found reductions guide the branch-and-bound search
into the wrong direction or make it more difficult for \texttt{SCIP} to
find other reductions.

\begin{table}
  \begin{scriptsize}
    \caption{Run times for different settings for \basic instances using the compact model.}
    \label{tab:compact}
    \begin{tabular*}{\textwidth}{@{}c@{\;\;\extracolsep{\fill}}ccrrrrrr@{}}\toprule
      \multicolumn{3}{c}{setting} & \multicolumn{2}{c}{cube} & \multicolumn{2}{c}{cross} & \multicolumn{2}{c}{simplex}\\
      \cmidrule{1-3} \cmidrule{4-5} \cmidrule{6-7} \cmidrule{8-9}
      hiding & sym. & prop. & time & \#solved & time & \#solved & time & \#solved\\
      \midrule
0 & 0 & 0 & \num{598.7} & \num{  13} & \num{1317.2} & \num{   8} & \num{232.8} & \num{  14}\\
0 & 0 & 1 & \num{795.3} & \num{  13} & \num{1395.4} & \num{   8} & \num{237.3} & \num{  14}\\
0 & a & 0 & \num{347.3} & \num{  14} & \num{476.7} & \num{  12} & \num{165.5} & \num{  14}\\
0 & s & 0 & \num{217.7} & \num{  15} & \num{283.3} & \num{  13} & \num{118.3} & \num{  15}\\
1 & 0 & 0 & \num{303.6} & \num{  13} & \num{682.0} & \num{  10} & \num{ 69.9} & \num{  15}\\
1 & a & 0 & \num{ 95.2} & \num{  15} & \num{206.1} & \num{  14} & \num{ 49.9} & \num{  16}\\
1 & a & 1 & \num{ 84.8} & \num{  15} & \num{221.5} & \num{  14} & \num{ 67.0} & \num{  16}\\
1 & s & 0 & \num{ 75.9} & \num{  18} & \num{151.0} & \num{  15} & \num{ 61.1} & \num{  16}\\
1 & s & 1 & \num{ 76.9} & \num{  18} & \num{158.4} & \num{  16} & \num{ 64.6} & \num{  16}\\
      \bottomrule
    \end{tabular*}
  \end{scriptsize}
\end{table}

Combining simple symmetry handling and hiding set cuts leads consistently
to the best results, reducing mean running time for cube instances
by~\SI{87}{\percent}, for crosspolytope instances by~\SI{89}{\percent}, and
simplex instances by~\SI{74}{\percent}.
In particular, the combined setting can solve all cube instances and almost
all crosspolytope and simplex instances within the time limit.

Next, we discuss the column generation model for which we only compare two
variants: we either disable or enable hiding set cuts in the pricing
problem.
Since the convexity propagator does not seem to be helpful for the compact
model, we do not enable it when solving the pricing problem.
Moreover, symmetry handling is not important, because there is only one
inequality to be identified by the pricing model.

\begin{table}
  \begin{scriptsize}
    \caption{Run times for different settings for \basic instances using the column generation model.}
    \label{tab:cg}
    \begin{tabular*}{\textwidth}{@{}c@{\;\;\extracolsep{\fill}}ccrrrrrr@{}}\toprule
      \multicolumn{3}{c}{setting} & \multicolumn{2}{c}{cube} & \multicolumn{2}{c}{cross} & \multicolumn{2}{c}{simplex}\\
      \cmidrule{1-3} \cmidrule{4-5} \cmidrule{6-7} \cmidrule{8-9}
      hiding & sym. & prop. & time & \#solved & time & \#solved & time & \#solved\\
      \midrule
0 & 0 & 0 & \num{109.0} & \num{  15} & \num{350.4} & \num{  13} & \num{455.2} & \num{  13}\\
1 & 0 & 0 & \num{ 73.5} & \num{  14} & \num{282.1} & \num{  12} & \num{312.7} & \num{  14}\\
      \bottomrule
    \end{tabular*}
  \end{scriptsize}
\end{table}

Comparing the column generation model with disabled hiding set cuts, we can
see that it performs for cube and crosspolytope instances much better than
the plain compact model:
the running time for cubes reduces by~\SI{82}{\percent} and for cross
polytopes by~\SI{74}{\percent}.
For cubes, all solvable instances are solved within the root node which is,
on the one hand, because of the strong dual bound as described in
Proposition~\ref{prop:qualityCG}.
On the other hand, the generated sets~$I \in \I$ allow heuristics to find
high quality solutions yielding a matching primal bound.
For crosspolytopes, all instances of 3-dimensional sets~$X$ can be solved
within the root node; for 4- and 5-dimensional sets, however, \texttt{SCIP}
needs to start branching to find an optimal solution.
Looking onto results on a per-instance basis reveals that the pricing
problems become considerably harder if~$d$ and~$k$ increases.
For example, \texttt{SCIP} is only able to process~2 nodes of the
branch-and-bound tree for~$d=k=5$.
For the simplex instances, the column generation model needs approximately
twice as much time as the plain compact model, which is again explained by
the very high running time of the pricing problem.

Enabling also hiding set cuts helps to solve the pricing problems more
efficiently.
In comparison with the enhanced compact model, however, the enhanced column
generation model is only competitive on the cube instances.
On the cross\-polytope and simplex instances, it is much slower.

\begin{table}
  \begin{scriptsize}
    \caption{Run times for different settings for \basic instances using the cut model.}
    \label{tab:cut}
    \begin{tabular*}{\textwidth}{@{}c@{\;\;\extracolsep{\fill}}ccrrrrrr@{}}\toprule
      \multicolumn{3}{c}{setting} & \multicolumn{2}{c}{cube} & \multicolumn{2}{c}{cross} & \multicolumn{2}{c}{simplex}\\
      \cmidrule{1-3} \cmidrule{4-5} \cmidrule{6-7} \cmidrule{8-9}
      hiding & sym. & prop. & time & \#solved & time & \#solved & time & \#solved\\
      \midrule
0 & 0 & 0 & \num{9229.1} & \num{   7} & \num{9411.0} & \num{   3} & \num{3099.1} & \num{   7}\\
0 & 0 & 1 & \num{9122.2} & \num{   7} & \num{9361.5} & \num{   3} & \num{3072.6} & \num{   7}\\
0 & a & 0 & \num{2549.9} & \num{   8} & \num{3734.6} & \num{   9} & \num{1726.5} & \num{  10}\\
0 & s & 0 & \num{4206.1} & \num{   7} & \num{5488.5} & \num{   6} & \num{2129.0} & \num{   8}\\
1 & 0 & 0 & \num{1733.0} & \num{   8} & \num{994.6} & \num{   7} & \num{428.5} & \num{  12}\\
1 & a & 0 & \num{528.0} & \num{  10} & \num{382.9} & \num{  10} & \num{257.8} & \num{  11}\\
1 & a & 1 & \num{424.8} & \num{  13} & \num{350.7} & \num{  10} & \num{259.8} & \num{  11}\\
1 & s & 0 & \num{433.6} & \num{  10} & \num{435.9} & \num{  11} & \num{296.3} & \num{  11}\\
1 & s & 1 & \num{435.9} & \num{  10} & \num{378.2} & \num{  12} & \num{277.0} & \num{  12}\\
      \bottomrule
    \end{tabular*}
  \end{scriptsize}
\end{table}

Finally, we consider the cutting plane model.
In the plain version, this model can hardly solve any instance efficiently.
Comparing the different enhancements with each other, we can see, analogously
to the compact model, that adding hiding set cuts and handling symmetries
is beneficial.
Interestingly, the convexity propagator helps to improve the running time
if both the previous enhancements are enabled
by~\SI{90}{\percent}--\SI{95}{\percent}, leading to the best setting for
this model.
But even this winner setting cannot compete with the enhanced compact
model.

From the results using the compact and cutting plane model, we draw the
following conclusion regarding the convexity propagator.
In principle, this method models the important aspect that the points being
cut by an inequality form a lattice-convex set.
The cutting plane method can thus benefit from the propagator as this
property is not encoded in the model.
The compact model, however, makes use of additional variables modeling the
inequalities of a relaxation.
Since the convexity propagator does not improve \texttt{SCIP}'s performance, we
conclude that these additional variables already sufficiently encode the
lattice-convexity of cut points.

In summary, the column generation model provides very good primal and dual
bounds.
If these bounds match, $\rc_\varepsilon(X,Y)$ can be computed rather
efficiently if not too many pricing problems need to be solved.
However, if the bounds do not match, the NP-hardness of the pricing
problem strikes back and solving many further pricing problems is too
expensive.
In this case, the compact model is a rather effective alternative that also
allows to compute~$\rc_\varepsilon(X,Y)$ for~$d=5$ in many cases.

\subsection{Results for Test Set \textsf{downcld}}

In this section, we turn the focus on~5-dimensional 0/1 down-closed sets.
On the one hand, our aim is to investigate whether the findings of the
previous section carry over to a much broader test set in dimension~5.
On the other hand, we are interested in identifying further sets~$X
\subseteq \bin{5}$ with~$\rc_\varepsilon(X,Y) < \rc_\Q(X)$ for every choice of a finite set~$Y \subseteq \Z^d$ and~$\varepsilon > 0$ small enough.
Because of our results on the \basic test set, we did not run any
experiments using the cutting plane model as we expect that it can
hardly solve any instance.
Instead, we consider a hybrid version of the compact model and the column
generation model:
We only solve the column generation model's LP relaxation to derive a
strong lower bound on~$\rc_\varepsilon(X,Y)$ and to find good primal
solutions.
Both are transferred to the compact model with the hope to drastically
reduce solving time.
The running times and number of nodes reported for the hybrid model
are means of the total running time and total number of nodes for solving
the LP relaxation in the column generation model and the resulting compact
model.

Table~\ref{tab:knapsack} shows aggregated results for the~99 instances of
the \dclosed test set for different~$\ell_1$-neighborhoods~$Y$ of~$X$
(radius~1--3).
While the plain compact model is able to solve two third of all instances
for radius~1, computing~$\rc_\varepsilon(X, Y)$ for larger radii becomes
much harder.
As the plain model can hardly solve any instance for radius at least~2,
there is definitively a need for model enhancements.
In general, the same observations as in the previous section can be made:
symmetry handling and adding hiding set cuts improve the solution process a
lot.
The biggest impact is achieved by symmetry handling; the convexity propagator
is not helpful in the best setting.
However, sometimes it can improve the running time, e.g., if the ``wrong''
symmetry handling method is used.

\afterpage{%
  \clearpage%
  \begin{landscape}
    \begin{table}
      \vspace{2cm}
      \begin{scriptsize}
        \caption{Comparison of running times for different settings for \dclosed instances.}
        \label{tab:knapsack}
        \begin{tabular*}{\linewidth}{@{}c@{\;\;\extracolsep{\fill}}ccrrrrrrrrr@{}}\toprule
          \multicolumn{3}{c}{setting} & \multicolumn{3}{c}{radius 1} & \multicolumn{3}{c}{radius 2} & \multicolumn{3}{c}{radius 3}\\
          \cmidrule{1-3} \cmidrule{4-6} \cmidrule{7-9} \cmidrule{10-12}
          hiding & sym. & prop. & \#solved & \#nodes & time & \#solved & \#nodes & time & \#solved & \#nodes & time\\
          \midrule
          \multicolumn{12}{l}{compact model:}\\
          0 & 0 & 0 & 69 & \num{351640.3} & \num{3168.5} & \num{5} & \num{149820.2} & \num{13173.5} & \num{0} & \num{29710.1} & \num{14400.0}\\
          0 & 0 & 1 & 69 & \num{342597.2} & \num{3155.1} & \num{3} & \num{145207.9} & \num{13237.5} & \num{0} & \num{29539.2} & \num{14400.0}\\
          0 & a & 0 & 97 & \num{59896.3} & \num{ 569.7} & \num{18} & \num{201253.6} & \num{12041.9} & \num{2} & \num{37865.0} & \num{14213.4}\\
          0 & s & 0 & 99 & \num{27421.4} & \num{ 187.8} & \num{58} & \num{278139.2} & \num{6299.6} & \num{3} & \num{83758.4} & \num{14033.7}\\
          1 & 0 & 0 & 82 & \num{24247.7} & \num{ 399.3} & \num{34} & \num{103717.2} & \num{7077.9} & \num{2} & \num{13704.6} & \num{14056.9}\\
          1 & a & 0 & 99 & \num{ 424.0} & \num{  22.6} & \num{75} & \num{16209.4} & \num{1870.2} & \num{14} & \num{14882.7} & \num{11990.5}\\
          1 & a & 1 & 99 & \num{ 424.4} & \num{  22.6} & \num{78} & \num{14936.4} & \num{1680.6} & \num{14} & \num{15975.0} & \num{11858.9}\\
          1 & s & 0 & 99 & \num{ 793.3} & \num{  30.5} & \num{99} & \num{8558.6} & \num{ 774.2} & \num{69} & \num{21696.0} & \num{6563.0}\\
          1 & s & 1 & 99 & \num{ 792.6} & \num{  30.5} & \num{99} & \num{8789.0} & \num{ 783.0} & \num{68} & \num{24044.2} & \num{7172.8}\\
          \midrule
          \multicolumn{12}{l}{column generation model:}\\
          0 &  &  & 78 & \num{  87.7} & \num{1223.6} & \num{4} & \num{  19.8} & \num{13424.8} & \num{1} & \num{   8.4} & \num{14089.9}\\
          1 &  &  & 85 & \num{ 109.3} & \num{ 644.5} & \num{6} & \num{  48.9} & \num{11500.4} & \num{1} & \num{  14.5} & \num{13949.5}\\
          \midrule
          \multicolumn{12}{l}{hybrid model:}\\
          0 & s & 0 & 99 & \num{8081.0} & \num{  89.0} & \num{85} & \num{375957.5} & \num{2865.8} & \num{15} & \num{341632.4} & \num{11644.5}\\
          0 & s & 1 & 99 & \num{8377.1} & \num{  99.2} & \num{84} & \num{375309.0} & \num{3406.0} & \num{11} & \num{377211.9} & \num{11959.0}\\
          1 & s & 0 & 99 & \num{ 347.1} & \num{  22.8} & \num{99} & \num{4219.6} & \num{ 339.3} & \num{80} & \num{18162.5} & \num{3203.7}\\
          1 & s & 1 & 99 & \num{ 346.2} & \num{  22.8} & \num{99} & \num{4437.3} & \num{ 347.0} & \num{79} & \num{18278.0} & \num{3123.0}\\
          \bottomrule
        \end{tabular*}
      \end{scriptsize}
    \end{table}
  \end{landscape}
}

For radius~2 and~3, we find that the simple symmetry handling methods
perform much better than the advanced methods.
Using hiding set cuts and simple symmetry handling is~\SI{59}{\percent}
faster than the corresponding setting with advanced symmetry handling if
the radius is~2; for radius~3, it is~\SI{45}{\percent} faster.
Moreover, simple symmetry handling can solve all~99 instances for radius~2
(resp.\ 69 instances for radius~3), whereas the advanced
setting can only solve~75 (resp.~14) instances.
Interestingly, for radius~1, the advanced setting is~\SI{26}{\percent}
faster than the simple setting.
A possible explanation is based on the nature of the advanced
setting:
Each inequality defining a relaxation of~$X$ w.r.t.~$Y$ defines a pattern
on the points from~$Y$ that are cut by this inequality.
The advanced method enforces that the cut patterns of the inequalities are
sorted lexicographically based on a sorting of the elements of~$Y$.
Since the results of the lexicographic comparison is determined by the first position in
which two vectors differ, it is unlikely that points having a late
position in the ordering of~$Y$ are very relevant for the lexicographic
constraint.
Thus, the symmetries are in a certain sense mostly handled for the early
points in this ordering.
In contrast to this, the simple method takes the geometry of the
inequalities in a relaxation into account by sorting inequalities based on
their first coefficients.
Together with other components of the solver, this seems to have more
implications on the cut points from~$Y$ if the radius becomes larger.

In comparison to the enhanced compact model, the column generation model is
again inferior.
For radius at least 2, it can hardly solve any instance and, as already
discussed in the previous section, the reason for this is the long running
time of the pricing models that need to be solved often at each node of the
tree.
This is reflected by the number of processed nodes during the
branch-and-price procedure that drops drastically (as the number of solved
instances) if the radius is getting larger.
However, we can again observe that the root node can be solved relatively
efficiently and that the obtained primal and dual bounds are rather strong.
This is reflected in the hybrid model, which solves most instances and
reduces the running time (in comparison to the best compact model)
by~52--\SI{56}{\percent} for radius~2 and~3.
For radius~1, the running times are comparable.

Regarding the usefulness of hiding set cuts in the hybrid model, we observe
that they are essential for solving the \dclosed instances efficiently.
They allow to solve all instances for radius~1 and 2 and improve on the
hybrid setting without cuts by~\SI{74}{\percent} and~\SI{88}{\percent},
respectively.
This effect is even more dominant for radius~3, where it significantly
increases the number of solvable instances, reducing the running time
by~\SI{72}{\percent}.
It is also noteworthy that the hybrid setting with hiding set cuts is the
only setting allowing to solve~80 instances, which improves the running time of
the compact model by~\SI{51}{\percent}.
In summary, based on our experiments, the hybrid model is the best choice
for computing~$\rc_\varepsilon(X,Y)$ as it combines the strong bounds from
the column generation model with the ability of the compact model to
quickly solve LP relaxations within the branch-and-bound tree.
In particular, it benefits from hiding set cuts since their implications
are very difficult to be found by \texttt{SCIP}.

Finally, concerning our goal to identify candidates for sets~$X \subseteq
\bin{5}$ such that~$\rc_\varepsilon(X,Y) < \rc_\Q(X)$ for all finite~$Y
\subseteq \Z^5$ and $\varepsilon>0$ small enough, our experiments for radius~3 revealed the
following:
If~$Y(X)$ are the integer points in the~$\ell_1$-neighborhood of~$X$ with
radius~3, then there are three sets~$X$ such that~$\rc_\varepsilon(X,Y(X))
= 4$.
These sets are~$\Delta_5$, $\Delta_5 \cup \{e_1 + e_2\}$ and~$\Delta_4
\times \{0,1\}$.
Moreover, there are~16 sets~$X$ with~$\rc_\varepsilon(X, Y(X)) = 5$.
It is left open for future research to identify which other sets than~$\Delta_5$ satisfy~$\rc_\varepsilon(X, Y(X)) < \rc_\Q(X)$.
Note that $\rc_\Q(X) \geq 6$, whenever $X \subseteq \bin{5}$ is full-dimensional, because rational relaxations must be bounded.

\subsection{Results for Test Set \sboxes}

The results for the \sboxes test set are summarized in Table~\ref{tab:sbox}.
Note that we do not report on results for the 10-dimensional instances in
the compact model with enabled hiding set cuts, because all these
experiments hit a memory limit of \SI{20}{\giga\byte}.
The reason is that these models grow very large even without any
enhancements as we use the number of facets of~$\conv(X)$ to upper
bound~$\rc_\varepsilon(X,Y)$; the number of facets for these instances
ranges between~888 and~2395.
For the largest instances, even the basic compact model hits the memory
limit.
Adding hiding set cuts for the remaining instances causes that all
instances hit the memory limit.
But also for the smaller instances, \texttt{SCIP} is hardly able to solve
any of these instances even if model enhancements are enabled due to huge
number of variables and constraints.

\begin{table}
  \vspace{2cm}
  \begin{scriptsize}
    \caption{Comparison of running times for different settings for \sboxes instances.}
    \label{tab:sbox}
    \begin{tabular*}{\linewidth}{@{}c@{\;\;\extracolsep{\fill}}ccrrrrrr@{}}\toprule
      \multicolumn{3}{c}{setting} & \multicolumn{3}{c}{dimension 8} & \multicolumn{3}{c}{dimension 10}\\
      \cmidrule{1-3} \cmidrule{4-6} \cmidrule{7-9}
      hiding & sym. & prop. & \#solved & \#nodes & time & \#solved & \#nodes & time\\
      \midrule
      \multicolumn{9}{l}{compact model:}\\
0 & 0 & 0 & 0 & \num{31942.1} & \num{14400.0} & \num{0} & \num{   4.1} & \num{14400.0}\\
0 & 0 & 1 & 0 & \num{31947.4} & \num{14400.0} & \num{0} & \num{   4.1} & \num{14400.0}\\
0 & a & 0 & 0 & \num{5497.3} & \num{14400.0} & \num{3} & \num{   6.6} & \num{12046.1}\\
0 & s & 0 & 0 & \num{3598.9} & \num{14400.0} & \num{2} & \num{   6.0} & \num{12492.2}\\
1 & 0 & 0 & 0 & \num{2970.5} & \num{14400.0} &  ---    & ---          & ---          \\
1 & a & 0 & 0 & \num{ 177.5} & \num{14400.0} &  ---    & ---          & ---          \\
1 & s & 0 & 0 & \num{ 163.5} & \num{14400.0} &  ---    & ---          & ---          \\
      \midrule
      \multicolumn{9}{l}{column generation model:}\\
0 &  &  & 12 & \num{  41.7} & \num{ 185.7} & \num{2} & \num{   7.5} & \num{6254.4}\\
1 &  &  & 12 & \num{  61.8} & \num{ 358.2} & \num{1} & \num{   3.8} & \num{14238.1}\\
      \midrule
      \multicolumn{9}{l}{hybrid model:}\\
0 & s & 0 & 7 & \num{6192.5} & \num{ 400.8} & \num{1} & \num{7471.5} & \num{7153.2}\\
0 & s & 1 & 7 & \num{6180.4} & \num{ 398.9} & \num{1} & \num{7513.8} & \num{7153.0}\\
1 & s & 0 & 9 & \num{14551.0} & \num{1251.3} & \num{1} & \num{  56.0} & \num{14054.3}\\
1 & s & 1 & 9 & \num{14551.0} & \num{1243.6} & \num{1} & \num{  56.0} & \num{14053.4}\\
      \bottomrule
    \end{tabular*}
  \end{scriptsize}
\end{table}

In contrast to this, we see that the column generation model performs
extremely well for the problems in dimension~8.
It can solve all twelve 8-dimen\-sional instances within the time limit, on
average in~\SI{185.7}{\second} if hiding set cuts are disabled and in
roughly twice this amount of time with enabled hiding set cuts.
An explanation for the worse behavior with enabled cuts is that the number
of hiding set cuts increases drastically in comparison with lower
dimensional problems.
Thus, creating and separating these cuts is a non-trivial task.
For dimension~10, the column generation model is also able to solve~2 out
of~6 instances within the time limit.

Finally, the hybrid model performs worse than the column generation model.
Although the derived bounds from solving the column generation model's LP
relaxation yield again very good bounds on the relaxation complexity, the
value of~$\rc(X, \bin{d})$ can still be large if~$d \in \{8,
10\}$.
Thus, also the compact model embedded in the hybrid model is struggling
with the number of variables and constraints.
For this reason, computing~$\rc(X, \bin{d})$ via the column
generation model is most competitive.

\subsection{Conclusions}

Being able to compute the exact value of the
quantity~$\rc_\varepsilon(X,Y)$ is highly relevant in many areas, such as,
social choice, symmetric cryptanalysis, or machine learning.
For this reason, we have proposed three different models that allow to
compute~$\rc_\varepsilon(X,Y)$ using mixed-integer programming techniques.
As our experiments reveal, each of these models comes with advantages and
disadvantages.
The compact model, for example, works well in small dimensions as the
number of variables and inequalities is small and it encapsulates all
essential information about~$\rc_\varepsilon(X,Y)$.
In higher dimensions, however, the dual bounds of the compact model become
weaker.
In this case, the column generation model provides very good bounds that
can be transferred to the compact model to still
compute~$\rc_\varepsilon(X,Y)$ rather efficiently if~$d=5$.
But if the dimension~$d$ grows even larger, only the column generation
model seems to be competitive as it does not scale as badly as the compact
model when~$\rc_\varepsilon(X,Y)$ increases.
The main reason is that the compact model is relying on a good upper bound
on~$\rc_\varepsilon(X,Y)$ to be indeed compact.

These findings thus open the following directions for future research.
Since the compact model requires a good upper bound
on~$\rc_\varepsilon(X,Y)$, it is natural to investigate heuristic
approaches for finding $\varepsilon$-relaxations of~$X$ or to develop
approximation algorithms.
Moreover, since the column generation model becomes more relevant if~$d$ is
large, it is essential that the pricing problem can be solved efficiently.
Since the pricing problem is NP-hard, also here a possible future direction
could be to develop heuristics or approximation algorithms for solving it.
For both the compact and column generation model, hiding set cuts turned
out to be useful.
However, we are not aware of an efficient routine for generating these
cutting planes.
Thus, it is natural to devise an efficient scheme for generating hiding set
cuts on the fly.
Finally, as additional inequalities such as hiding set cuts and symmetry
handling inequalities drastically improved the performance of the compact
model, the development of further inequalities modeling structural
properties of relaxation complexity might allow to solve the compact model
even more efficiently.
\medskip

\textbf{Acknowledgements}
We thank Aleksei Udovenko for providing the \sboxes instances used by him
in~\cite{Udovenko2021}.

\bibliographystyle{spmpsci}      

\end{document}